\documentclass[12pt]{article}
\usepackage{amsmath,amsxtra,amssymb,color,latexsym,amscd,amsthm,subfigure,fancybox}
\usepackage[mathscr]{eucal}
\usepackage{bbm}
\usepackage{graphicx}
\usepackage{enumerate}
\usepackage{enumitem}
\usepackage{epsfig} 
\usepackage{epstopdf}
\usepackage{diagbox, eqparbox, hhline}
\usepackage{cases}
\usepackage{float}
\newtheorem{thm}{Theorem}[section]

\newtheorem{lm}{Lemma}[section]

\theoremstyle{definition}

\theoremstyle{remark}

\numberwithin{equation}{section}

\newcommand{\eps}{\varepsilon}

\newcommand{\M}{\mathcal{M}}

\newcommand{\E}{\mathbb{E}}

\newcommand{\N}{{\mathbb{Z}}_+}

\newcommand{\PP}{\mathbb{P}}

\newcommand{\R}{\mathbb{R}}

\numberwithin{equation}{section}

\newcommand{\Lom}{{\mathcal L}}

\newcommand{\bed}{\begin{displaymath}}
\newcommand{\eed}{\end{displaymath}}
\newcommand{\bea}{\bed\begin{array}{rl}}
\newcommand{\eea}{\end{array}\eed}

\newcommand{\barray}{\begin{array}{ll}}
\newcommand{\earray}{\end{array}}

\def\bar{\overline}
\def\hat{\widehat}
\def\a.s{\text{\;a.s.\;}}

\def\bdelta{\boldsymbol{\delta}}
\begin{document}
\title{Extinction and permanence in a stochastic SIRS model  in regime-switching with general incidence rate\thanks{This
research was supported in part by the Vietnam National Foundation for Science and Technology Development   (NAFOSTED)
n$_0$  101.03-2017.23.}}
\date{}
\author{T. D. Tuong\thanks{Faculty of Basic Sciences, Ho Chi Minh University of Transport, 2 D3, Ho Chi Minh city, Vietnam; VNU, Hanoi-University of Science,
trandinhtuong@gmail.com.}\and
Dang H.  Nguyen\thanks{Department of Mathematics, Wayne State University, Detroit, MI
48202, United States,
dangnh.maths@gmail.com.} \and
N.T. Dieu\thanks{Department of Mathematics, Vinh University,
182 Le Duan, Vinh, Nghe An, Vietnam, dieunguyen2008@gmail.com.}\and
Ky Tran\thanks{Department of Mathematics, College of Education, Hue University, 34 Le Loi street, Hue city, Vietnam, quankysp@gmail.com.}}
\maketitle
\begin{abstract}
In this paper, we consider a stochastic SIRS model with general incidence rate and perturbed by both white noise and color noise.  We determine the threshold $\lambda$ that is used
to classify the extinction and permanence of the disease.  In particular,   $\lambda<0$ implies that the disease-free $(K, 0, 0)$ is globally asymptotic stable, i.e., the disease will eventually disappear. If $\lambda>0$ the epidemic is strongly stochastically permanent.
Our result is considered as a significant generalization and improvement over the results  in \cite{HZ1, GLM1, LOK1, SLJJ1, ZJ1}.
\end{abstract}

\section{Introduction}\label{sec:int}
In recent years, mathematical models have been used increasingly to support public
health policy making in the field of infectious disease control. The first roots
of mathematical modeling date back to the eighteenth century, when Bernoulli \cite{Ber}
used mathematical methods to estimate the impact of smallpox vaccination on life
expectancy. However, a rigorous mathematical framework was first worked out by Kermack and Mckendrick \cite{KM, KM1}. Their model, nowadays best known as the SIR model. This model classifies
individuals as one of susceptible, infectious and removed with permanent acquired
immunity. In fact, some removed individuals lose immunity and return to the susceptible
compartment. This case can be modeled by SIRS (Susceptible-Infected-Romoved-Susceptible) epidemic model, which  studied by
many scholars \cite[...]{HZ1, GLM1,  KM1, KW1,LOK1,  Lu,  SLJJ1, ZJ1}; see also \cite{Z.1, Z.2, Z.3} for related works. In fact,  the disease transmission process is unknown in detail. However, several authors proposed
different forms of incidences rate in order to model the disease transmission process. In  \cite{KW1}, authors studied deterministic SIRS model  with the standard bilinear incidence rate and has been extended to stochastic SIRS model in \cite{CLL1, HZ1, LS1, SLJJ1}. However, there is a variety of reasons why this standard bilinear incidence rate
may require modifications. In \cite{CS1}, Capasso and Serio studied the cholera epidemic spread in Bari in 1978. They imposed the saturated incidence rate $\frac{\beta SI}{1+aI}$ in their model of the cholera, where $a$ is positive constant. Anderson et. al. \cite{AM1} used saturated incidence rate $\frac{\beta SI}{1+aS}.$ In \cite{NT1}, authors considered the Beddington-DeAngelis functional response $\frac{\beta SI}{1+a_1S+a_2I}.$ 
Recently, there are many works on epidemic models perturbed by both white and colored noises, for example [10, 20, 25] and our work can be seen as a further step.
	The consideration of colored noise is
	motivated by 	
	the fact that biological parameters of systems often demonstrate abrupt changes which have important effects
	on the dynamics of  the system (see [26]); that is, there might be sudden instantaneous transitions between two or more sets of parameter values in
	the underlying model corresponding to two or more different environments or regimes such as changes between dry or raining times, types of pathogens, number of mediators.  The switching is often assumed memoryless and the waiting time for the next
	switch has an exponential distribution. We can hence model the regime switching by a finite-state
	Markov chain.  For instance,   Anderson [27] discusses the optimal
	exploitation strategies for an animal population in a Markovian environment. Additionally Peccoud and Ycart [28] propose
	a Markovian model for the gene induction process, and Caswell and Cohen [29] discuss the effects of the spectra of the
	environmental variation in the coexistence of metapopulations using a two-state Markov chain. 

In this paper, we work with the general incidence rate $SIF_1(S, I),$ where $F_1$ is locally Lipschitz continuous. Thus, our model includes 
incidence rates appeared above. 
  Furthermore, we suppose that the model is perturbed by both white nose and color noise. 
  To be specific, we consider the following model  
\begin{equation}\label{e1.1b}
\begin{cases}
dS(t)=\big(-S(t)I(t)F_1(S(t), I(t),\xi_t)+\mu(\xi_t) (K-S(t))+\gamma_1(\xi_t)R(t)\big)dt\\
\hspace{8.5cm}-S(t)I(t)F_2(S(t), I(t),  \xi_t)dB(t)  \\
dI(t)=\big(S(t)I(t)F_1(S(t), I(t),\xi_t)-(\mu(\xi_t) +\rho(\xi_t)+ \gamma_2(\xi_t)) I(t))dt\\ \hspace{8.5cm}+S(t)I(t)F_2( S(t), I(t),\xi_t)dB(t)  \\
dR(t)=(\gamma_2(\xi_t) I(t) - (\mu(\xi_t)+\gamma_1(\xi_t)) R(t))dt,
\end{cases}
\end{equation}
where $\{\xi_t, t\geq 0\}$ is a right continuous Markov chain taking values in $\mathcal M=\{1,2,\dots,m_0\}$,
$F_1(\cdot), F_2(\cdot)$ are positive and locally  Lipchitz functions on $[0,\infty)^{2}\times\M$,  $B(t)$ is a one dimensional Brownian motion, all  parameters $K$, $\mu(i)$, $\rho(i)$, $\gamma_1(i)$, $\gamma_2(i)$ are assumed to be positive for all $i\in\mathcal M$. $K$ is  a carrying capacity, $\mu(i)$, $\rho(i)$, $\gamma_1(i)$, $\gamma_2(i)$ 
are the per capita disease-free death rate, the excess per capita natural death rate of  infective class,    the per capita lose immunity and return to the sucesstible class of  infective class,  and the per capita recovery rate of the infective individuals respectively in the $ith$ regime.

Our main goal in this paper is to  provide a sufficient and almost necessary condition for  strongly stochastically permanent and extinction of the disease in the stochastic SIRS model \eqref{e1.1b}. Concretely, we establish a threshold $\lambda$ such that the sign of $\lambda$ determines the asymptotic behavior of the system.
 If $\lambda < 0,$ the disease is eradicated at a disease-free equilibrium $(K, 0, 0)$. In this case, we derive that the density of disease converges to $0$ with exponential rate.  Meanwhile, in the case $\lambda > 0$, by using techniques in \cite{BL1} we show that the disease is strongly stochastically permanent.
 Compared to existing results in  
 \cite{HZ1, GLM1, LOK1, SLJJ1, ZJ1}, our findings are significant improvements as we will show in Section \ref{sec:con}. 
 We emphasize that the main method in the aforementioned papers is based on the construction of Lyapunov functions. Meanwhile, we adopt a new approach motivated by the works \cite{Dieu, Gray} on SIS and SIR models. 
 Therefore, this work paves an effective way for treating generalized SIRS models in random environments. 
 
\par
 The rest of the paper is arranged as follows. In Section \ref{sec:thr}, we derive a threshold that is used
to classify the extinction and strongly stochastically permanent of the disease. A large part of this section is devoted to the proof of Theorem \ref{main}.
A discussion and comparison to
 existing results in the literature together with several numerical experiments are presented in Section \ref{sec:con}. The paper is concluded with some additional remarks.

\section{Sufficient and almost necessary conditions for permanence}\label{sec:thr}
Let  $B(t)$ be an one-dimensional
 Brownian motion defined on a complete probability space $(\Omega,\mathcal{F},\mathbb{P} )$.  Denote by $Q=({q_{kl}})_{m_0\times m_0}$  the generator of the Markov chain $\{\xi_t,t\geq 0\}$  taking values in $\mathcal M=\{1,2,\dots,m_0\}$ for a positive integer $m_0$.
This means that
\begin{equation*}
\mathbb P\{\xi_{t+\delta}=l|\xi_t=k\}=
\begin{cases}
q_{kl}\delta + o(\delta)& \text{if   } k\ne l,\\
1+q_{kk}\delta+ o(\delta) &\text{if   } k=l,
\end{cases}
\end{equation*}
{as } $\delta\to 0$. Here, $q_{kl}$ is the transition rate from $k$ to $l$ and $q_{kl}\geq 0$ if $k\ne l$, while
$q_{kk}=-\sum_{k\ne l}q_{kl}.$
We assume that the Markov chain $\xi_t$ is irreducible, which means
that the system will switch from any regime to any other regimes.
Under this condition, the Markov chain $\xi_t$ has a unique stationary distribution $\pi=(\pi_1,\pi_2,\ldots,\pi_{m_0})\in\mathbb{R}^{m_0}$.

We assume that the Markov chain $\xi_t$ is independent of the Brownian motion $B(t)$. Denote $\mathbb{R}^{3}_+:=\{(x,y,z): x\geq 0, y\geq 0, z\geq 0\}$ and $\Delta:=\{(x,y,z)\in\R^3_+: x+y+z\leq K\}.$
The interior $\{(x,y,z): x> 0, y> 0, z> 0\}$ of $\mathbb{R}^{3}_+$ is denoted by $\mathbb{R}^{3, o}_+$. Throughout of this paper, we suppose that $F_j(x, y, i)>0$ for all $(x, y,z, i)\in \Delta\times \mathcal{M}, \; j=\overline{1,2}.$
\begin{thm}
For any given initial value $(S(0), I(0), R(0))\in\mathbb{R}^{3}_+$,  there exists a unique global solution ${ \{(S(t), I(t), R(t)), t\geq0\}}$ of Equation \eqref{e1.1b} and the solution will remain in  $\mathbb{R}^{3}_+$ with probability one. Moreover, if $I(0)>0$ then $I(t)>0$ for any $t\geq 0$ with probability one.

\end{thm}
\begin{proof}
The proof is almost the same as those in  \cite{HZ1}.
\end{proof}
By adding side by side in system \eqref{e1.1b},  we have
\begin{align*}
  \frac d{dt}(S(t)+I(t)+R(t))
&=K\mu(\xi_t)-\mu(\xi_t)(S(t)+I(t)+R(t))-\rho(\xi_t) I(t)
\\& \leq  \mu(\xi_t)K-\mu(\xi_t)(S(t)+I(t)+R(t)) .
  \end{align*}
  In view of the comparison theorem, if $S(0)+I(0)+R(0)\leq K$, so is $(S(t)+I(t)+R(t))$ for $t\geq0$. Thus, $\Delta=\{(x,y,z)\in\R^3_+: x+y+z\leq K\}$ is an invariant set.
  If $S(0)+I(0)+R(0)\leq K$, 
  then $\limsup_{t\to\infty}(S(t)+I(t)+R(t))\leq K$. 
Therefore, we only need to work with the process $(S(t),I(t), R(t))$ on the invariant set $\Delta$.
To simplify notations,
we denote by $\Phi(t)=(S(t), I(t), R(t))$  the solution of system \eqref{e1.1b},  and $\phi=(x, y, z)\in \Delta$.
We are now in a position to  provide a condition for the extinction and permanence of disease.
 Let $$g(x,y,i)=F_1(x,y,i)x-\Big(\mu(i)+\rho(i)+\gamma_2(i)+\frac{F_2^2(x, y, i)x^2}{2}\Big).$$ 
We define the threshold \begin{equation}\label{lambda} \lambda =\sum_{i=1}^{m_0}g(K,0,i)\pi_i= \sum_{i=1}^{m_0} \Big[F_1(K, 0, i) K-\Big(\mu(i)+\rho(i)+\gamma_2(i)+\frac{F_2^2(K, 0, i)K^2}{2}\Big)\Big]\pi_i. \end{equation} 
 Let $C^{2}(\R^3\times\mathcal M, \R_+)$  denote the family of all non-negative functions $V(\phi, i)$
on $\R^3\times\mathcal M $ which are twice continuously differentiable in $\phi$. The operator $\Lom$ associated with \eqref{e1.1b} is defined as follows. 
For $V\in C^{2}(\R^3\times\mathcal M, \R_+)$, define
\begin{equation}
\Lom V(\phi,i)=\Lom_i V(\phi, i)+\sum_{j\in\M} q_{ij}V(\phi, j),
\end{equation}
where $$
\Lom_i V(\phi, i)=V_\phi(\phi, i)\widetilde f(\phi, i)+\frac12\widetilde g^\top(\phi, i) V_{\phi\phi}(\phi, i)\widetilde g(\phi, i),
$$ 
$V_\phi(\phi, i) $ and $V_{\phi\phi}(\phi, i)$ are the gradient and Hessian of $V(\cdot,i)$,
$\widetilde f$ and $\widetilde g$ are the drift and diffusion coefficients of \eqref{e1.1b}, respectively; i.e.,
\begin{equation}\label{2.3}
\begin{aligned}
\widetilde f(\phi,i)=&(-xyF_1(x,y,i)+\mu(i)(K-x)+\gamma_1(i)z, xyF_1(x,y,i)\\
&\qquad -(\mu(i)+\rho(i)+\gamma_2(i))y, \gamma_2(i)y-(\mu(i)+\gamma_1(i))z)^\top,
\end{aligned}
\end{equation}
 and $$\widetilde g(\phi,i)=(-xyF_2(x,y,i), xyF_2(x,y,i),0)^\top.$$ 
Note that if $V(\phi,i)$ does not depend on $i$, then $\Lom_iV$ and $\Lom V$ coincide because $\sum_{j\in\M} q_{ij}=0$.

Our main result is given below.
\begin{thm}\label{main}
 If $\lambda<0$, then $\Phi(t)\to (K,0, 0)$ a.s. as $t\to\infty$
for all given initial value $(\phi, i)\in \Delta\times \mathcal M$, i.e., the disease will be extinct. Moreover,
\begin{equation}\label{log}
\PP_{\phi,i}\Big\{\lim\limits_{t\to\infty}\dfrac{\ln I(t)}t=\lambda<0\Big\}=1\text{ for }(\phi,i)\in\Delta\times\M, y>0.
\end{equation}
 If $\lambda>0$,  the disease is strongly stochastically permanent
in the sense that for any $\eps>0$, there exists a $\delta>0$ such that
\begin{equation}\label{per}
\liminf_{t\to\infty} \PP_{\phi,i}\{I(t)\geq\delta\}>1-\eps\,\text{ for any }\, (\phi,i)\in\Delta\times\M, y>0.
\end{equation}
\end{thm}
\begin{proof}[Proof of Theorem \ref{main}
	: Case $\lambda<0$.]
Since $\lambda<0$, we can choose sufficiently small $\sigma>0$ such that
$$\sum_{j\in\M}(g(K,0,j)+\sigma)\pi_j<0.$$
Consider the Lyapunov function $V(x,y,z, i)=(K- x)^2+y^p+z^2$, where $p\in (0, 1)$ is a constant to be specified. By direct calculation we have
for $(x,y,z,i)\in\Delta\times\M$ that
 \begin{align*} \mathcal{L}_iV(x,y, z, i)=&-2(K-x)[-F_1(x, y, i) xy+\mu(i)(K- x)+\gamma_1(i)z] +py^{p}g(x,y,i)
\\&+2z(\gamma_2(i)y-(\mu(i)+\gamma_1(i)) z) +x^2y^2F_2^2(x, y, i) +\frac{p^2F^2_2(x, y, i)x^2y^p}{2}
\\\leq& -2\mu(i)(K-x)^2 -2(\mu(i)+\gamma_1(i))z^2+py^{p}g(x,y,i)
 \\&
 + y\Big(2(K-x)F_1(x, y, i) x+2z\gamma_2(i)+x^2yF_2^2(x, y, i)\Big)+\frac{p^2F^2_2(x, y, i)x^2y^p}{2}.
 \end{align*}

For a constant $\delta_1\in (0, K)$, we denote $\mathcal{U}_{\delta_1}=(K-\delta_1, K]\times[0, \delta_1)^2.$
  Because of the continuity of $g(\cdot), F_1(\cdot), F_2(\cdot)$, the compactness of $\Delta\times\M$ and the fact that $y^{1-p}\to0$ as $y\to 0$, we can choose $p\in (0,1)$ and $\delta_1 \in (0, K)$ such that for any $(x, y, z, i)\in \mathcal{U}_{\delta_1}\times\mathcal{M}$,
  \begin{align*} py^pg(x,y,i)+y\Big(2(K-x)F_1(x, y, i) x&+2z\gamma_2(i)+x^2yF_2^2(x, y, i)\Big)+\frac{p^2F^2_2(x, y, i)x^2y^p}{2}
\\& \leq p(g(K,0,i)+\sigma)y^p.
\end{align*}
When $p$ is sufficiently small, we also have $$-2\mu(i)(K-x)^2 -2(\mu(i)+\gamma_1(i))z^2\leq p(g(K, 0, i)+\sigma)[(K-x)^2+z^2].$$
  Therefore, $$\Lom_iV(x, y, z, i)\leq p[g(K, 0, i)+\sigma]V(x, y, z, i) \,\text{ for }\, (x, y, z, i)\in \mathcal{U}_{\delta_1}\times \mathcal{M}.$$
  By \cite[Theorem 3.4]{DY1} (see also \cite[Definition 3.1]{KZY} and \cite[Theorem 4.3]{KZY}), for any $\varepsilon>0,$ there is $0<\delta<\delta_1$ such that
\begin{equation}\label{e2.19s}
\PP_{\phi,i}\Big\{\lim_{t\to\infty}(S(t), I(t), R(t))=\big(K,0, 0\big)\Big\}\geq 1-\varepsilon \, \text{ for }\, (\phi,i)\in \mathcal{U}_{\delta}\times\mathcal{M}.
 \end{equation}
Now we show that any solution starting in $\Delta\times\M$
will eventually enter $\mathcal{U}_{\delta}\times\mathcal{M}$.
Let $\tau_\delta=\inf\{t\geq 0: S(t)\geq K-\delta\}$. Consider the Lyapunov function 
$U(\phi, i)=c_1-(x+1)^{c_2},$
where $c_1$ and $c_2$ are two positive constants to be specified.
We have
$$
\mathcal{L}U(\phi, i)=-c_2(x+1)^{c_2-2}[(x+1)(\mu(i)(K-x)+\gamma_1(i)z-xyF_1(x, y, i))
+\frac{c_2-1}2x^2y^2F_2^2(x, y, i)].
$$

Let $\mu_m=\min\{\mu(i):i\in\M\}$.
Since
$(x+1)\mu(i)(K-x)\geq\delta\mu_m$ for any $x\in[0,K-\delta]$
and $\inf\{F_2(x,y,i): (x,y,z,i)\in\Delta\times\M\}>0$,
we can find a sufficiently large number $c_2$ such that
\begin{equation}
(x+1)\mu(i)(K-x)+\gamma_1(i)z-xyF_1(x, y, i))
+\frac{c_2-1}2x^2y^2F_2(x, y, i)\geq 0.5\delta\mu_m 
\end{equation}
for $(\phi,i)\in\Delta\times\M, x\leq K-\delta$. Then
$$\mathcal{L}U(\phi, i)\leq-0.5c_2\delta\mu_m\;\;\text{given that }\;(x,y,z,i)\in\Delta\times\M, x\leq K-\delta.$$
By Dynkin's formula, we obtain
$$\E_{\phi,i} U(\Phi(\tau_\delta\wedge t), r_{\tau_\delta\wedge t})= U(\phi,i)+\E_{\phi,i}\int_0^{\tau_\delta\wedge t}\Lom U(\Phi(s),r_s)ds\leq U(\phi, i)-0.5 c_2\mu_m\delta\E_{\phi,i}\tau_{\delta}\wedge t.$$
Letting $t\to\infty$ and using Fatou's lemma yields that
$$\E_{\phi,i} U(\Phi(\tau_\delta), \xi_{\tau_\delta})\leq U(\phi, i)-0.5\delta\mu_m c_2\E_{\phi,i}\tau_{\delta}.$$
Since $U$ is bounded above on $\R^3_+$, we deduce that
$\E_{\phi,i}\tau_{\delta}<\infty$.
By the strong Markov property, we have from \eqref{e2.19s} and $\E_{\phi,i}\tau_{\delta}<\infty$ that
$$\PP_{\phi,i}\{\lim_{t\to\infty}\Phi(t)=(K,0,0)\}\ge (1-\eps)\,\text{ for }\, (\phi, i)\in \Delta\times\M,$$
for any $\eps>0$. As a result,
\begin{equation}\label{e2.5dl1}
\PP_{\phi,i}\{\lim_{t\to\infty}\Phi(t)=(K,0,0)\}=1\,\text{ for }\, (\phi, i)\in \Delta\times\M.
\end{equation}
By It\^o's formula
we have
$$\ln I(t)=\ln I(0)- G(t),$$
where
$$G(t)=-\int_0^tg(\Phi(u),\xi_u)du-\int_0^tS(u)I(u)F_2(S(u), I(u),  \xi_u)dB(u).
$$This imlies that
\begin{equation}\label{e2.6dl1}
\dfrac{\ln  I(t)}t=\frac{\ln I(0)}t+\dfrac1t\int_0^tg(\Phi(u),\xi_u)du
+\dfrac1t\int_0^t S(u)I(u)F_2(S(u), I(u),  \xi_u)dB(u).
\end{equation}
By the strong law of large numbers for martingales and ergodic Markov processes,
we derive from   \eqref{lambda} and \eqref{e2.5dl1} that
$$
\lim_{t\to\infty}\dfrac1t\int_0^tg(\Phi(u),\xi_s)du=\lambda\,\text{ and }\,
\lim_{t\to\infty}\dfrac1t\int_0^t S(u)I(u)F_2(S(u), I(u),  \xi_u)dB(u)=0\,\text{ a.s}.
$$
This and \eqref{e2.6dl1} imply \eqref{log}.
\end{proof}

To prove the permanence of the species when $\lambda>0$,
we use the techniques in \cite{BL1}.
We need some lemmas.

\begin{lm} Let $\partial \Delta_2:=\{\phi=(x,y,z)\in\Delta: y=0\}$. Then
there exists $T>0$ such that
for any $(\phi,i)\in \partial \Delta_2\times\M$,
\begin{equation}\label{e3.1}
\E_{\phi,i}\int_0^Tg(\Phi(t),\xi_t)dt\geq \dfrac{3\lambda}4T.
\end{equation}
\end{lm}
\begin{proof}
When $I(0)=0$, we have $I(t)=0$ for any $t>0$ and $(S(t), R(t))$ satisfies
\begin{equation*}
\begin{cases}
dS(t)=\big(\mu(\xi_t) (K-S(t))+\gamma_2(\xi_t)R(t)\big)dt\\
dR(t)=- (\mu(\xi_t)+\gamma_2(\xi_t)) R(t))dt.
\end{cases}
\end{equation*}
It is easy to see that $(S(t),R(t))$ converges to $(K,0)$ uniformly in the initial values.
This and the uniform ergodicity of $\xi_t$ imply that
$$\lim_{t\to\infty}\dfrac1t\E_{\phi,i}\int_0^tg(\Phi(u),\xi_u)du=\lambda\text{ uniformly in }(\phi,i)\in\partial\Delta_2\times\M.$$
Thus, we can easily find a constant $T$ satisfying \eqref{e3.1}.
\end{proof}
\begin{lm}\label{laplace}
Let $Y$ be a random variable, suppose $\E \exp(Y)+\E \exp(-Y)\leq K_1.$
Then the log-Laplace transform
$u(\theta)=\ln\E\exp(\theta Y)$
is twice differentiable on $\left[0,0.5\right]$ and
$\frac{du}{d\theta}(0)= \E Y,$
$0\leq \frac{d^2u}{d\theta^2}(\theta)\leq 2K_2\,, \theta\in\left[0,0.5\right]$
 for some $K_2>0$ depending only on $K_1$. Thus, it follows from Taylor's expansion that
 $$u(\theta)\leq \E Y\theta+K_2\theta^2, \, \theta\in\left[0,0.5\right].$$
\end{lm}

\begin{proof} The proof of this lemma can be found in \cite{DY1}. For convenience, we present a sketch of the proof below. 
	It is easy to show that there exists some $K_2>0$ such that
	$$ |y|^k\exp(\theta y)\leq K_2(\exp(y)+\exp(-y)), k=1,2.$$
	for $\theta\in \left[0,\frac{1}2\right]$, $y\in\R$.
	For any $y\in\R$, let $\xi(y)$ be a number lying between $y$ and $0$ such that	
	$\exp(\xi(y))=\dfrac{e^y-1}y$.
	Pick $\theta\in\left[0,\frac{1}2\right]$
	and let $h\in\R$ such that $0\leq \theta+h\leq  \frac{1}2$. Then
	$$\lim\limits_{h\to0}\dfrac{\exp((\theta+h) Y)-\exp(\theta Y)}h= Y\exp(\theta Y)\text{ a.s.,} \text{ and}$$
	$$\left|\dfrac{\exp((\theta+h) Y)-\exp(\theta Y)}h\right|=|Y|\exp(\theta Y+\xi(hY))
	\leq 2K_3[ \exp(Y)+\exp(-Y)].$$
	By the Lebesgue dominated convergence theorem,
	$$\dfrac{d \E \exp(\theta Y)}{d\theta}=\lim\limits_{h\to0}\E\dfrac{\exp((\theta+h) Y)-\exp(\theta Y)}h= \E Y\exp(\theta Y).$$
	Similarly,
	$$\dfrac{d^2 \E \exp(\theta Y)}{d\theta^2}=\E Y^2\exp(\theta Y).$$
	As a result, we obtain
	$$\dfrac{d\phi}{d\theta}=\dfrac{\E Y\exp(\theta Y)}{\E \exp(\theta Y)}$$
	which implies
	$$\dfrac{d\phi}{d\theta}(0)=\E Y \quad 
	\text{and}\quad 
	\dfrac{d^2\phi}{d\theta^2}=\dfrac{\E Y^2\exp(\theta Y)\E \exp(\theta Y)-[\E Y\exp(\theta Y)]^2}{[\E \exp(\theta Y)]^2}.$$
	By H{\"o}lder's inequality we have $\E Y^2\exp(\theta Y)\E \exp(\theta Y)\geq[\E Y\exp(\theta Y)]^2$
	and therefore $$\dfrac{d^2\phi}{d\theta^2}\geq0\quad \text{for all}\quad \theta\in \left[0, 0.5\right].$$
	Moreover,
	$$
	\begin{aligned}
	\dfrac{d^2\phi}{d\theta^2}\leq& \dfrac{\E Y^2\exp(\theta Y)}{\E \exp(\theta Y)}
	\leq  \dfrac{K_3(\E \exp(Y)+\E \exp(- Y))}{\exp(\theta \E Y)}\\
	\leq & \dfrac{K_3(\E \exp(Y)+\E \exp(- Y))}{\exp(-|\E Y|)}:=K_2,
	\end{aligned}
	$$
	which concludes the proof.
\end{proof}
 
 We proceed to prove Theorem \ref{main} in the case that $\lambda>0$.

\begin{proof}[Proof of Theorem \ref{main}. Case: $\lambda>0$]
Consider the Lyapunov function  $V_\theta(\phi,i)=y^\theta$, where $\theta$ is a real constant to be determined. We have
$$
\Lom_iV_\theta(\phi, i)=\theta y^\theta[F_1(x, y, i)x-(\mu(i)+\rho(i)+\gamma_2(i))+\frac{\theta-1}2x^2F_2^2(x, y, i)].
$$
 It implies that $L V_\theta\leq H_\theta V_\theta$, where $$H_\theta=\sup\{\theta[F_1(x, y, i)x-(\mu(i)+\rho(i)+\gamma_2(i))+\frac{\theta-1}2x^2F_2^2(x, y, i)]: (x, y, z, i)\in \Delta\times \mathcal{M}\}.$$
Thus, by using It\^o's formula and taking expectation both sides, we obtain
\begin{equation}\label{e-itheta}
\E_{\phi,i} I^\theta(t)\leq y^\theta \exp(H_\theta t)\text{ for any }t\geq0, (\phi,i)\in(\Delta\setminus\partial\Delta_2)\times\M.
\end{equation}
By the Feller property and \eqref{e3.1},
there exists $\delta_2>0$ such that if $\phi=(x, y, z)\in\Delta$ with $y<\delta_2$
we have
\begin{equation}\label{e3.2}
\E_{\phi,i}G(T)=-\E_{\phi,i}\int_0^Tg(\Phi(t),\xi_t)dt\leq
-\dfrac{\lambda}2T.
\end{equation}
Since $G(t)=\ln I(0)-\ln I(t)$, 
we
have from \eqref{e-itheta} that $E_{\phi,i}(e^{G(t)}+e^{G(t)})\leq e^{H_1t}+e^{H_{-1}t}<\infty$.
Applying Lemma \ref{laplace},
we deduce from \eqref{e3.2} that
$$
\ln\E_{\phi,i}e^{\theta G(T)}\leq -\dfrac{\lambda\theta}2T+\hat H\theta^2 \text{ for } \theta\in[0,0.5],
$$
where $\hat H$ is a constant depending on $T$, $H_{-1}$ and $H_1$.
For sufficiently small $\theta$, we have

$$\E_{\phi,i}\frac{y^{\theta}}{I^{\theta}(T)}=\E_{\phi,i} \frac{I^{\theta}(0)}{I^{\theta}(T)}=
\E_{\phi,i}e^{\theta G(T)}\leq \exp(-\dfrac{\lambda\theta}4T)\,\text{ for } \phi\in\Delta,y<\delta_3, i\in\M.
$$

Equivalently,
$$\E_{\phi,i}I^{-\theta}(T)\leq q y^{-\theta} \,\text{ for }\, q= \exp(-\dfrac{\lambda\theta}4T)\,\text{ for } \phi\in\Delta,y<\delta_3,i\in\M.$$
This and \eqref{e-itheta} imply that
$$\E_{\phi,i}I^{-\theta}(T)\leq q y^{-\theta} +C\,\text{ for }\, C= \delta_3^{-\theta}\exp(H_{-\theta} T)\,\text{ for } \phi\in\Delta,i\in\M.$$
By the Markov property, we deduce that
$$\E_{\phi,i}I^{-\theta}((k+1)T)\leq q \E_{\phi,i}I^{-\theta}(kT) +C\,\text{ for } \phi\in\Delta,i\in\M, k\in\N.$$
Using this recursively we obtain
\begin{equation}
\E_{\phi,i}I^{-\theta}(nT)
\leq q^n y^{-\theta} +\dfrac{C(1-q^n)}{1-q}\,\text{ for } \phi\in\Delta,i\in\M, n\in\N.
\end{equation}
This estimate together with \eqref{e-itheta} leads to
\begin{equation}\label{e1-thm3.1}
\E_{\phi,i}I^{-\theta}(t)
\leq \left(q^n y^{-\theta} +\dfrac{C(1-q^n)}{1-q}\right)\exp(H_{-\theta} T)\,\text{ for } t\in[nT, nT+T].
\end{equation}
Letting $n\to\infty$ we obtain
$\limsup_{t\to\infty} \E_{\phi,i}I^{-\theta}(t)=\dfrac{C}{1-q}\exp(H_{-\theta} T),$
which implies \eqref{per}. The proof is thus completed.
\end{proof}
\section{Discussion and Numerical Experiments}\label{sec:con}

To highlight the contributions of this work, we compare our results with some of the recent developments in the literature. In fact, \cite{SLJJ1} considered the model
\begin{equation}\label{EX8}
\begin{cases}
dS(t)=\big(\mu(\xi_t)-\beta(\xi_t)S(t)I(t)-\mu(\xi_t)S(t)+\gamma(\xi_t)R(t)\\
\qquad \qquad\qquad \qquad -\sigma(\xi_t)S(t)I(t)(S(t)+I(t))\big)dt-\sigma(\xi_t)S(t)I(t)dB(t),  \\
dI(t)=\big(\beta(\xi_t)S(t)I(t)-(\mu(\xi_t)+\lambda(\xi_t))I(t)\big)dt+ \sigma(\xi_t)S(t)I(t)dB(t), \\
dR(t)=\big(\lambda(\xi_t) I(t) - (\mu(\xi_t)+\gamma(\xi_t)) R(t)\big)dt.
\end{cases}
\end{equation}

In that paper, they showed that

\begin{thm}\label{rs-20}
	\begin{enumerate}
	\item
	If $\beta_j\geq\sigma_j^2$ and $\sum\pi_jC_j<0$
	then the disease-free equilibrium is globally asymptotically stable in probability,
	where $C_j=\beta_j-\mu_j-\lambda_j-\dfrac12\sigma_j^2$.
	\item If
	$
	\sum\pi_j\left(\dfrac{\beta_j^2-2\mu_j\sigma_j^2}{2\sigma_j^2}\right)<0
	$
	then  the disease-free equilibrium is globally asymptotically stable almost surely.
	\item If  $\sum\pi_jC_j>0$ then the disease persists.
	\end{enumerate}
\end{thm}
On the other hand,
for this model, our $\lambda$ is determined by
$\lambda=\sum_{j\in\M} \pi_j(\beta_j-\mu_j-\lambda_j-\dfrac12\sigma_j^2)=\sum\pi_jC_j$
 an application of our results reads
that if $\lambda=\sum\pi_jC_j<0$,  the disease-free equilibrium is globally asymptotically stable almost surely.
In case $\lambda=\sum\pi_jC_j>0$, the disease persists.
Thus, our findings provide sharper results for the extinction of the disease
because we do not need the additional condition that $\beta_j\geq\sigma_j^2$ as in (1) of Theorem \ref{rs-20}.
Moreover, since $\beta_j-\frac{\sigma_j^2}2\leq \frac{\beta^2_j}{2\sigma_j^2}$ (following Cauchy's inequality $\beta_j\leq \frac{\beta^2_j}{2\sigma_j^2}+\frac{\sigma_j^2}2)
$
we have 
$$C_j=\beta_j-\mu_j-\lambda_j-\dfrac12\sigma_j^2\leq \dfrac{\beta_j^2-2\mu_j\sigma_j^2}{2\sigma_j^2}-\lambda_j< \dfrac{\beta_j^2-2\mu_j\sigma_j^2}{2\sigma_j^2}$$ 
which shows that the condition in (2) of Theorem \ref{rs-20} is much more restrictive  than ours.

In \cite{ZJ1}, they considered the model
\begin{equation}\label{EX17}
\begin{cases}
dS(t)=\big(\Lambda-\mu S(t)-\frac{\beta S(t)I(t)}{1+\alpha I(t)}+\delta R(t)\big)dt-\frac{\sigma S(t)I(t)}{1+\alpha I(t)}dB(t),  \\
dI(t)=\big(\frac{\beta S(t)I(t)}{1+\alpha I(t)}-(\mu + \gamma+\varepsilon) I(t)\big)dt+\frac{\sigma S(t)I(t)}{1+\alpha I(t)}dB(t),\\
dR(t)=\big(\gamma I(t) - (\mu+\delta) R(t)\big)dt.
\end{cases}
\end{equation}
They proved that if
$$\tilde R_0:=\dfrac{\beta \Lambda}{\mu(\mu+\gamma+\eps)}-\dfrac{\sigma^2\Lambda^2}{\mu^2(\mu+\gamma+\eps)}>1,
$$
the system is persistent in time-average.
For \eqref{EX17}, our threshold 
$$\lambda:=\beta\frac{\Lambda}{\mu}-(\mu+\gamma+\eps)-\frac{\sigma^2\Lambda^2}{\mu^2}=(\bar R_0-1)(\mu+\gamma+\eps)$$

Thus,  our theorem shows the persistence in probability which is stronger then persistence in time-average as $\lambda>0$ or equivalently $\bar R_0>1$. 
 Regarding to the extinction, we provide a more  relaxing condition. More specifically,
our results read that if $\lambda<0$ (or equivalently $\tilde R_0<1$) then 
the disease goes extinct with probability one, while
the condition for extinction in 
\cite{ZJ1} is
either $$\text{(a) } \sigma^2>\dfrac{\beta^2}{2(\mu+\gamma+\eps)} \quad \text{or (b) } \tilde R_0<1 \text{ and } \sigma^2\leq\dfrac{\beta\mu}{\Lambda}.$$
 Clearly, under either (a) or (b), we have $\lambda<0$ (the inverse implication is not true),
which implies that our condition for extinction $\lambda<0$ is more relaxing.

Focusing on a stochastic SIRS model with regime-switching, we have determined a threshold value whose sign specifies
whether or not the disease goes to extinct or survive permanently.

Working with a general incidence rate and a taking into account both white noise and color noise,
the model includes almost all SIRS models appeared in the literature (e.g. \cite{HZ1, LOK1,  NT1, Cai, Guo}).
In this paper, a nearly full classification for the asymptotic behaviors of the model
has been given. Only the critial case when $\lambda=0$ is not studied yet.
We also provide the exact exponential convergence rate when $\lambda<0$
which is not obtained using existing methods.
In constrast, in most existing results,
besides a threshold, some additional conditions are needed
in order to obtain the extinction and/or the permanence of the disease.
As a result,
our findings can be seen as significant extensions of results in the aforementioned papers. Moreover, the method we have used suggests an effective approach in treating SIRS systems.

\end{document}